\documentclass[11pt]{amsart}

\usepackage{amsthm}
\usepackage[all,pdf,cmtip]{xy}
\usepackage{amssymb,amsfonts,amsmath}
\usepackage{amscd}
\usepackage[english]{babel}
\usepackage{import}
\usepackage{mathtools}
\usepackage[colorlinks=true]{hyperref}
\usepackage{subfiles}
\usepackage{amsmath,amssymb,amscd,amsfonts,graphicx,latexsym,epsfig,psfrag,url,color,xr,mathtools,enumitem,transparent,fullpage,hyperref,float,mathrsfs,wrapfig,rotating,array}
\usepackage[foot]{amsaddr}

\newcommand\disp{\displaystyle}

\newtheorem{theorem}{Theorem}[section]

\newtheorem{proposition}[theorem]{Proposition}

\newtheorem{defprop}[theorem]{Definition-Proposition}

\theoremstyle{definition}
\newtheorem{definition}[theorem]{Definition}
\theoremstyle{remark}
\newtheorem{remark}[theorem]{Remark}

\theoremstyle{plain}
\newcommand{\thistheoremname}{}
\newtheorem{genericthm}[theorem]{\thistheoremname}

\newtheorem*{genericthm*}{\thistheoremname}
\newenvironment{namedthm*}[1]
  {\renewcommand{\thistheoremname}{#1}%
   \begin{genericthm*}}
  {\end{genericthm*}}
  
\newcommand\cC{\mathcal{C}}
\newcommand\cD{\mathcal{D}}

\newcommand\cM{\mathcal{M}}

\newcommand\Tree{\mathcal{TR}}

\newcommand{\bR}{\mathbb{R}}

\newcommand{\bZ}{\mathbb{Z}}

\newcommand\bm{\mathbf{m}}
\newcommand\bn{\mathbf{n}}
\newcommand\bp{\mathbf{p}}
\newcommand\bs{\mathbf{s}}

\newcommand\bzero{\mathbf{0}}

\newcommand{\on}{\operatorname}

\newcommand\pt{{\on{pt}}}
\newcommand\pr{\on{pr}}

\newcommand\id{\on{id}}

\newcommand{\Fuk}{\on{Fuk}}
\newcommand{\comp}{C^2}
\newcommand{\Symp}{\mathsf{Symp}}

\newcommand{\op}{{\on{op}}}

\renewcommand{\comp}{{\on{comp}}}
\newcommand{\seam}{{\on{seam}}}
\newcommand{\mk}{{\on{mark}}}
\newcommand{\incom}{\on{in}}

\newcommand{\inte}{{\on{int}}}

\newcommand{\tree}{{\on{tree}}}

\newcommand{\proj}{{\on{proj}}}

\newcommand{\oneO}{{\on{Cube}}}
\newcommand{\twoO}{{\on{2Cube}}}
\newcommand{\Top}{{\textsf{Top}}}

\newcommand{\Ob}{\on{Ob}}
\newcommand{\Mor}{\on{Mor}}

\newcommand{\spn}{\on{span}}
\newcommand{\coll}{\on{coll}}
\newcommand{\height}{\on{ht}}

\newcommand\qu{/\kern-.7ex/} 
\newcommand\lqu{\backslash \kern-.7ex \backslash}

\newcommand{\ol}{\overline}

\newcommand{\sr}{\stackrel}

\newcommand{\wt}{\widetilde}

\newcommand{\eps}{\epsilon}

\def\hra{\hookrightarrow}
\def\lra{\longrightarrow}

\newcounter{qcounter}

\newcommand\quotient[2]{
        \mathchoice
            {
                \text{\raise1ex\hbox{$#1$}\Big/\lower1ex\hbox{$#2$}}%
            }
            {
                #1\,/\,#2
            }
            {
                #1\,/\,#2
            }
            {
                #1\,/\,#2
            }
    }

\newcommand\quoti[2]{
                \text{\raise1ex\hbox{$#1$}/\lower1ex\hbox{$\scriptstyle#2$}}
  }

\newcommand\quot[2]{
                \text{\raise1ex\hbox{$#1\!\!$}/\lower1ex\hbox{$\!\scriptstyle#2$}}
  }

\newcommand\quo[2]{
                \text{\raise.8ex\hbox{$\scriptstyle#1\!$}/\lower.8ex\hbox{$\!\scriptstyle#2$}}
  }

\newcommand\qq[2]{
                \text{\raise.8ex\hbox{$#1\!$}/\lower.8ex\hbox{$#2$}}
}

\title{$(A_\infty,2)$-categories and relative 2-operads}

\author{Nathaniel Bottman}
\address{Max Planck Institute for Mathematics,
Vivatsgasse 7, 53111 Bonn, Germany}
\address{Department of Mathematics, University of Southern California,
3620 S Vermont Ave, Los Angeles, CA 90089, USA}
\email{bottman@mpim-bonn.mpg.de}

\author{Shachar Carmeli}
\address{Faculty of Mathematics and Computer Science, Weizmann Institute, 234 Herzl Street, Rehovot 7610001, Israel}
\email{shachar.carmeli@weizmann.ac.il}

\begin{document}

\maketitle

\begin{abstract}
We define the notion of a 2-operad relative to an operad, and prove that the 2-associahedra form a 2-operad relative to the associahedra.
Using this structure, we define the notions of an $(A_\infty,2)$-category and $(A_\infty,2)$-algebra in spaces and in chain complexes over a ring.
Finally, we show that for any continuous map $A \to X$, we can associate the related notion of an $\wt{(A_\infty,2)}$-algebra $\theta(A \to X)$ in \Top, which specializes to $\theta(\pt \to X) = \Omega^2 X$ and $\theta(A \to \pt) = \Omega A \times \Omega A$.
\end{abstract}

\section{Introduction}

\label{s:intro}

The first author recently constructed in \cite{b:2ass} a family of abstract polytopes called 2-associahedra, which he realized as stratified spaces in \cite{b:realization}.
These spaces are intended to play the same role as associahedra do for the definition of an $A_\infty$-category, but for a new algebraic notion called an {$(A_\infty,2)$-category}.
As developed in \cite{b:2ass,b:realization,b:sing,bw:compactness,b:thesis,mww}, the correct way to express the functoriality of the Fukaya category is to construct an $(A_\infty,2)$-category (the \emph{symplectic $(A_\infty,2)$-category}) whose objects are symplectic manifolds and where $\hom(M,N)$ is $\Fuk(M^-\times N)$.
The definition of the notion of an $(A_\infty,2)$-category is therefore a fundamental step in the first author's ongoing project to construct the symplectic $(A_\infty,2)$-category.

In this paper we show that the 2-associahedra (or their realizations) have an operad-like structure: they form a \emph{2-operad relative to the associahedra}.
The notion of a relative 2-operad as such is new.
It can be phrased in terms of Batanin's theory of higher operads (see \S\ref{ss:batanin} and Prop.~\ref{prop:connection_with_batanin}), but we feel that the concept of a relative 2-operad is natural enough to deserve its own name and definition.
One can define a category over a relative 2-operad, and when we specialize to the 2-operad of topologically-realized 2-associahedra, we obtain the definition of $(A_\infty,2)$-categories in \Top.

Recall from \cite{b:2ass} that for every $r \geq 1$ and $\bn \in \bZ_{\geq0}^r\setminus\{\bzero\}$ there is a 2-associahedron $W_\bn$, which is an abstract polytope (after adding a formal minimal element) of dimension $|\bn|+r-3$ and, in particular, a poset.
In \cite{b:realization}, the first author constructed realizations of the 2-associahedra in terms of witch curves, denoted $\ol{2\cM}_\bn$; $\ol{2\cM}_\bn$ is a compact metrizable space stratified by $W_\bn$.
These realizations satisfy the following properties, which inspire our Def.~\ref{def:2op} of a relative 2-operad:
\begin{itemize}
\item[] \textsc{(forgetful)} $\ol{2\cM}_\bn$ is equipped with a forgetful map $\pi\colon \ol{2\cM}_\bn \to \ol\cM_r$ to the compactified moduli space of disks with $r+1$ boundary marked points, which is a continuous and surjective map of stratified spaces.

\item[] \textsc{(recursive)} For any stable tree-pair $2T = (T_b \sr{p}{\to} T_s) \in W_\bn^\tree$, there is a continuous and injective map of stratified spaces
\begin{align}
\Gamma_{2T} \colon \prod_{
\substack{\alpha \in V_\comp^1(T_b),
\\
\incom(\alpha)=(\beta)}}
\ol{2\cM}_{\#\!\incom(\beta)}^\tree
\times
\prod_{\rho \in V_\inte(T_s)} \prod^{\ol\cM_{\#\!\incom(\rho)}}_{
\substack{
\alpha\in
V_\comp^{\geq2}(T_b)\cap f^{-1}\{\rho\},
\\
\incom(\alpha)=(\beta_1,\ldots,\beta_{\#\!\incom(\rho)})}
}
\hspace{-0.25in} \ol{2\cM}^\tree_{\#\!\incom(\beta_1),\ldots,\#\!\incom(\beta_{\#\!\incom(\alpha)})}
\hra \ol{2\cM}_\bn^\tree,
\end{align}
where the superscript on one of the product symbols indicates that it is a fiber product with respect to the maps described in \textsc{(forgetful)}.
\end{itemize}

These ingredients allow us to state the first main result of this paper, taken from \S\ref{ss:examples}:

\medskip
\noindent {\bf Definition-Proposition~\ref{defprop:2Mn_is_2op}, paraphrased.}
{\it The realized 2-associahedra $(\ol{2\cM}_\bn)$, together with the forgetful maps $\pi\colon \ol{2\cM}_\bn \to \ol\cM_r$ and certain of the structure maps $\Gamma_{2T}$, form a 2-operad relative to the associahedra $(\ol\cM_r)$.
The same statement is true when $\ol{2\cM}_\bn$ resp.\ $\ol\cM_r$ are replaced by the 2-associahedra $W_\bn$ resp.\ $K_r$.}
\medskip

\noindent Indeed, these properties of the 2-associahedra and their realizations get to the heart of the definition of a relative 2-operad: such a thing consists of an underlying operad together with a collection of objects indexed by $\bigsqcup_{r\geq1} (\bZ_{\geq0}^r\setminus\{\bzero\})$, together with maps of the form $\pi$ and $\Gamma_{2T}$ that satisfy suitable compatibility conditions.
(More precisely, we only need structure maps $\Gamma_{2T}$ for certain tree-pairs, as described in Def.-Prop.~\ref{def:2op}.)
In the same subsection we give another example of a relative 2-operad, which is denoted $(\twoO_\bn)$ and is closely related to the little cubes family of operads.
We expect that $(\twoO_\bn)$ and $(\ol{2\cM}_\bn)$ are homotopy equivalent in an appropriate model categories of relative 2-operads, which we intend to define in the future.

Next, we define in \S\ref{s:cats_over_2ops} the notion of a \emph{category over a relative 2-operad in a category with finite limits} and an \emph{$R$-linear category over a relative 2-operad in \textsf{Top}}.
The latter definition allows us to make the following definition, which formed the first author's original motivation to formulate and study the 2-associahedra:

\medskip
\noindent
{\bf Definition~\ref{def:A_infty_2_cat}.}
An \emph{$R$-linear $(A_\infty,2)$-category} is an $R$-linear category over the relative 2-operad $\bigl((\ol\cM_r),(\ol{2\cM}_\bn)\bigr)$.
\null\hfill$\triangle$

\medskip

\noindent
We produce a family of examples of $\wt{(A_\infty,2)}$-algebras (i.e.\ $(\wt{A_\infty,2})$-categories with one object), where the tilde indicates that we work with $(\twoO_\bn)$, rather than $(\ol{2\cM}_\bn)$:

\medskip
\noindent
{\bf Proposition~\ref{prop:theta}.}
{\it Fix a map $f\colon (A,q) \to (X,p)$ of pointed topological spaces.
Define a space $\theta(A \to X)$ by
\begin{align}
\theta(A \to X)
\coloneqq
\left\{
\left.\left(
\substack{
u\colon [0,1]^2 \to X
\\
\gamma_\pm\colon [0,1] \to A}
\right)
\:\right|\:
\substack{
u(-,0) = f\circ\gamma_-,
\\
u(-,1) = f\circ\gamma_+},
\:
\substack{
u(0,-) = p = u(1,-)
\\
\gamma_\pm(0) = q = \gamma_\pm(1)}
\right\},
\end{align}
and equip $\theta(A\to X)$ with maps $s,t \colon \theta(A\to X) \to \Omega A$ that send $(u,\gamma_+,\gamma_-)$ to $\gamma_-$ resp.\ $\gamma_+$.
Then the pair $\theta(A\to X) \sr{s,t}{\rightrightarrows} \Omega A$ is an $\wt{(A_\infty,2)}$-algebra.}
\null\hfill$\triangle$
\medskip

We close this introduction by mentioning Michael Batanin's theory of $m$-operads, which is related to the notion of relative 2-operad defined in this paper.
In fact, relative 2-operads can be phrased as instances of 2-operads, as we show in Prop.~\ref{prop:connection_with_batanin}.
Moreover, in \cite{ba:config} Batanin proposed a collection of spaces $(B_T)$, where $T$ ranges over the 2-ordinals; this collection forms a 2-operad, and the spaces $B_T$ seem to be surjective images of the spaces $\ol{2\cM}_\bn$.
An example of this is shown in Fig.~\ref{fig:BT}.
\begin{figure}[H]
\centering
\def\svgwidth{0.6\columnwidth}
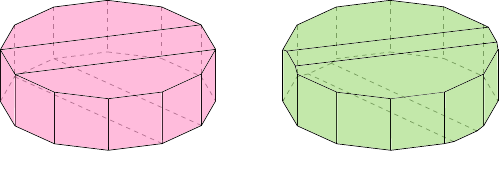
\caption{There are some similarities between the spaces $\ol{2\cM}_\bn$ and the spaces $B_T$ defined by Batanin in \cite{ba:config}, for $T$ a 2-ordinal.
(In fact, Batanin allows $T$ to be an $m$-ordinal.)
(a) is the space $B_T$ corresponding to the 2-ordinal $T \coloneqq 0 <_0 < 1 <_0 2$; (b) is the space $W_{111}$.
(a) can be obtained from (b) by collapsing 8 of the edges in (b) into vertices.}
\label{fig:BT}
\end{figure}

\subsection{Further directions}

\begin{itemize}
\item Symplectic geometers define an ($R$-linear) $A_\infty$-category to be a category over the operad of cellular chains on realized associahedra, with respect to the obvious cellular structure.
Ultimately, it would be convenient to have an analogous definition of $(A_\infty,2)$-categories, as opposed to the definition we give in this paper, which uses singular chains on realized 2-associahedra.
It is not currently clear to the authors how to accomplish this, because faces of 2-associahedra decompose canonically as products \emph{of fiber products} of 2-associahedra.
In future work we aim to address this issue.

\smallskip

\item It would be very interesting to understand the connection between the 2-associahedral relative 2-operad and the little 2-disks operad $E_2$.
(This would be related to Batanin's exploration of the connection between $E_n$-operads and $n$-operads, see in particular \cite{ba:config,ba:eckman-hilton,ba:monoidal}.)
Once this is accomplished, we hope that finding such a connection would shed light on what happens when one restricts an $(A_\infty,2)$-category to a single object with the identity 1-morphism.
One might speculate that such a restriction would have the structure of a homotopy Gerstenhaber algebra.
See \cite{b:fm} for some recent progress in this direction.

\smallskip

\item Categories over the $A_\infty$-operad are exactly $A_\infty$-categories in the sense of \cite{ba:homotopy}.
There is a homotopy theory for such catgegories and there is a natural nerve functor of $\infty$-categories from $A_\infty$-categories to the $(\infty,1)$-category $\textsf{Cat}_\infty$. We expect a similar picture in the case of $(A_\infty,2)$-categories: namely, that they can be organized to form an $\infty$-category using a model structure on them, and that there is a functor to the $(\infty,1)$-category of $(\infty,2)$-categories.
Then, one could consider the 2-associahedra as encoding higher coherences in certain $(\infty,2)$-categories in an economic way. 
\end{itemize}

\section{Relative 2-operads}

In this section we will define and explore the notion of a relative 2-operad.
We begin by defining this notion in \S\ref{ss:def_of_relative_2-operad}.
Next, in \S\ref{ss:examples} we give two examples of relative 2-operads: $(\twoO_\bn)$ is a 2-operad of spaces relative to the little intervals operad, and the 2-associahedra $(W_\bn)$ form a 2-operad of posets relative to the operad of associahedra; the latter statement is also true for the topological realizations in the category of spaces.
Finally, in the subsection \S\ref{ss:batanin} we analyze the relationship between relative 2-operads and Batanin's notion of 2-operads.

From now on, $\prod_i^Y X_i$ will denote the fiber product of a collection of objects $(X_i)$ in a category $\cC$ with respect to morphisms $X_i \to Y$.

\subsection{The definition of a relative 2-operad}
\label{ss:def_of_relative_2-operad}

Before we come to the definition of a relative 2-operad, we set notation by recalling the definition of an operad.

\begin{definition}[Def.~1.4, \cite{mss:operads}]
A \emph{non-$\Sigma$ operad} in a symmetric monoidal category $(\cC,\otimes,1)$ is a collection $(P_r)_{r\geq1} \subset \cC$ together with a family of structure morphisms
\begin{align}
\gamma_{r,(s_i)}\colon P_r \otimes \bigotimes_{1\leq i\leq r} P_i \to P_{\sum_i s_i},
\qquad
r, s_1, \ldots s_r \geq 1
\end{align}
satisfying the following axioms:
\begin{itemize}
\item[] {\sc(associative)} The following diagram commutes:
\begin{align}
\label{eq:operad_ass}
\xymatrix{
\disp{P_r \otimes \bigotimes_{1\leq i\leq r} P_{s_i} \otimes \bigotimes_{
\substack{
1\leq i\leq r
\\
1\leq j\leq s_i}}}
P_{t_{ij}} \ar[rr]^\simeq
\ar[dd]_{\gamma_{r,(s_i)}\times\id}
&&
\disp{P_r \otimes \bigotimes_{1\leq i\leq r} \Bigl(P_{s_i}\otimes\bigotimes_{1\leq j\leq s_i} P_{t_{ij}}\Bigr)}
\ar[d]^{\id\otimes\bigotimes_{1\leq i\leq r} \gamma_{s_i,(t_{ij})_j}}
\\
&&
\disp{P_r \otimes \prod_{1\leq i\leq r} P_{\sum_j t_{ij}}}
\ar[d]^{\gamma_{r,(\sum_j t_{ij})}}
\\
\disp{P_{\sum_i s_i} \otimes \bigotimes_{
\substack{1\leq i\leq r
\\
1\leq j\leq s_i}} P_{t_{ij}}
\ar[rr]_{\qquad\gamma_{\sum_i s_i, (t_{1,1},\ldots,t_{1,s_1},\ldots,t_{r,1},\ldots,t_{r,s_r})}}}
&&
\disp{P_{\sum_{i,j} t_{ij}}.}
}
\end{align}

\item[] {\sc(unit)} There is a \emph{unit map} $\eta\colon 1 \to P_1$ such that the compositions
\begin{align}
P_r \otimes 1^{\otimes r} \:\sr{\id \otimes \eta^{\otimes r}}{\lra}\: P_r \otimes P_1^{\otimes r} \:\sr{\gamma_{r,(1,\ldots,1)}}{\lra}\: P_r,
\qquad
1 \otimes P_s \:\sr{\eta\otimes\id}{\lra} P_1 \otimes P_s \:\sr{\gamma_{1,(s)}}{\lra}\: P_s
\end{align}
are the iterated right resp.\ left unit morphism in $\cC$.
\null\hfill$\triangle$
\end{itemize}
\end{definition}

\begin{defprop}
\label{defprop:Kr_is_an_operad}
For any $r\geq 1$ and $\bs \in \bZ_{\geq1}^r$, define $T_{r,(s_i)}$ to be the following element of $K_{\sum_i s_i}^\tree$:
\begin{figure}[H]
\centering
\def\svgwidth{0.1\columnwidth}
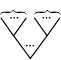
\end{figure}
\noindent Define $\gamma_{r,(s_i)}\colon K_r^\tree\times \prod_i K_{s_i}^\tree \to K_{\sum_i s_i}^\tree$ by setting $\gamma_{r,(s_i)} \coloneqq \gamma_{T_{r,(s_i)}}$, where the latter map was defined in Def.-Lem.~2.14, \cite{b:2ass}.
Then $(K_r)_{r \geq 1}$ with these composition maps forms a non-$\Sigma$ operad in the category of posets.
Similarly, $(\ol\cM_r)_r$ is a non-$\Sigma$ operad in \textsf{Top}.
\end{defprop}

\begin{proof}
To prove that the operations $\gamma_{r,(s_i)}\colon K_r^\tree \times \prod_i K_{s_i}^\tree \to K_{\sum_i s_i}^\tree$ make $(K_r)_{r\geq 1}$ into a non-$\Sigma$ operad, we must verify {\sc(associative)} and {\sc(unit)}.
{\sc(unit)} is an immediate consequence of the definitions of $\gamma_{1,(s)}$ and $\gamma_{r,(1,\ldots,1)}$.
{\sc(associative)} follows from a diagram chase:
\begin{figure}[H]
\centering
\def\svgwidth{1.0\columnwidth}
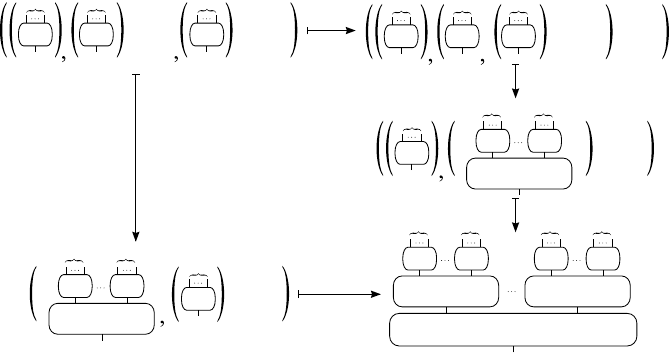
\end{figure}
\end{proof}

\begin{definition}
\label{def:2op}
A \emph{non-$\Sigma$ relative 2-operad} in a category $\cC$ with finite limits is a pair
\begin{align}
\Bigl((P_r)_{r\geq 1}, (Q_\bm)_{\bm \in \bZ^r_{\geq0}\setminus\{\bzero\},r \geq 1}\Bigr),
\end{align}
where $(P_r)_{r\geq 1}$ is a non-$\Sigma$ operad in $\cC$ with structure morphisms $\gamma_{r,(s_i)}$, and where $(Q_\bm) \subset \cC$ is a collection of objects together with a family of structure morphisms
\begin{align}
\label{eq:Gamma_def}
\Gamma_{\bm,(\bn^a_i)}
\colon
Q_\bm \times \prod_{1\leq i\leq r} \prod^{P_{s_i}}_{1\leq a\leq m_i} Q_{\bn_i^a} \to Q_{\sum_a \bn_1^a,\ldots,\sum_a \bn_r^a},
\quad
r, s_1, \ldots s_r \geq 1, \bm \in \bZ_{\geq0}^r\setminus\{\bzero\}, \bn_i^a \in \bZ_{\geq0}^{s_i}\setminus\{\bzero\}.
\end{align}
(Here the subscript in $Q_{\sum_a \bn_1^a,\ldots,\sum_a \bn_r^a}$ denotes the concatenation of $\sum_a \bn_1^a$, $\sum_a \bn_2^a$, etc., which is a vector of length $\sum_i s_i$.)
We require these objects and morphisms to satisfy the following axioms.
\begin{itemize}
\item[] {\sc(projections)} $\bigl((P_r),(Q_\bm)\bigr)$ is equipped with projection morphisms
\begin{align}
\pi_\bm\colon Q_\bm \to P_r,
\qquad
r \geq 1,
\:
\bm \in \bZ^r_{\geq0}\setminus\{\bzero\}
\end{align}
such that the following diagram commutes:

\begin{align}
\label{eq:Gamma_gamma_coherence}
\xymatrix{
\disp{Q_\bm \times \prod_{1 \leq i \leq r} \prod^{P_{s_i}}_{1\leq a \leq m_i} Q_{\bn_i^a}
\ar[dd]_{\bigl(\pi_\bm, \disp\prod_{1 \leq i \leq r}  \pi\bigr)}
\ar[rr]^{\hspace{0.2in}\Gamma_{\bm, (\bn_i^a)}}}
&&
\disp{Q_{\sum_j \bn_1^a, \ldots, \sum_a \bn_r^a} \ar[dd]^{\pi_{\sum_a \bn_1^a, \ldots, \sum_a \bn_r^a}}}
\\
&&
\\
\disp{P_r \times \prod_{1\leq i\leq r} P_{s_i}}
\ar[rr]_{\gamma_{r,(s_i)}}
&&
\disp{P_{\sum_i s_i}}.
}	
\end{align}

\medskip

\item[] {\sc(associative)} 
The following diagram commutes:
\begin{align}
\label{eq:relative_2-operad_Gamma}
\xymatrix{
\disp{Q_\bm \times \prod_{1\leq i\leq r} \prod_{1 \leq a \leq m_i}^{P_{s_i}} Q_{\bn_i^a} \times \prod_{
\substack{
1\leq i\leq r
\\
1\leq j\leq s_i}}
\prod_{
\substack{
1\leq a\leq m_i
\\
1\leq b \leq n_{ij}^a}}^{P_{t_{ij}}} Q_{\bp_{ij}^{ab}}}
\ar[r]^{\hspace{-0.2in}\simeq}
\ar[dd]_{\Gamma_{\bm,(\bn_i^a)}\times\id}
&
\disp{Q_\bm \times \prod_{1\leq i\leq r} \prod_{1\leq a\leq m_i}^{P_{s_i}\times\prod_j P_{t_{ij}}} \Bigl(Q_{\bn_i^a} \times \prod_{1\leq j\leq s_i}\prod_{1\leq b\leq n_{ij}^a}^{P_{t_{ij}}} Q_{\bp_{ij}^{ab}}\Bigr)}
\ar[d]^{\id \times \bigl(\Gamma_{\bn_i^a,(\bp_{ij}^{ab})}\bigr)}
\\
&
\disp{Q_\bm \times \prod_{1\leq i\leq r}\prod_{1\leq a\leq m_i}^{P_{\sum_j t_{ij}}} Q_{\sum_b \bp_{i1}^{ab}, \ldots, \bp_{is_i}^{ab}}}
\ar[d]^{\Gamma_{\bm,(\sum_b \bp_{i1}^{ab},\ldots,\sum_b \bp_{is_i}^{ab})}}
\\
\disp{Q_{\sum_a \bn_1^a,\ldots,\sum_a\bn_r^a} \times \prod_{
\substack{
1\leq i\leq r
\\
1\leq j\leq s_i}} \prod_{
\substack{
1\leq a\leq m_i
\\
1\leq b \leq n_{ij}^a}
}^{P_{t_{ij}}} Q_{\bp_{ij}^{ab}}}\qquad
\ar[r]_{\Gamma_{(\sum_a \bn_1^a,\ldots,\sum_a \bn_1^a),(\bp_{ij}^{ab})}\vspace{11in}}
&
\qquad\quad\disp{Q_{\sum_{a,b} \bp^{ab}_{11},\ldots,\sum_{a,b} \bp^{ab}_{1s_1},\ldots,\sum_{a,b} \bp^{ab}_{r1},\ldots,\sum_{a,b} \bp^{ab}_{rs_r}}},
}
\end{align}
where the fiber products are with respect to the projection morphisms described in {\sc(projections)}, and where the notation of concatenated sums appearing as subscripts should be interpreted in the same way as we explained in Def.~\ref{def:2op}.

\medskip

\item[] {\sc(unit)}
If $1$ denotes the final object of $\cC$, then there is a ``unit map" $\kappa\colon 1 \to Q_1$ such that the compositions
\begin{align}
Q_\bm \times 1^{\times |\bm|} \:\sr{\id \times \kappa^{\times \bn}}{\lra}\: Q_\bm \times Q_1^{\times |\bm|} \:\sr{\Gamma_{\bm,((1,\ldots,1),\ldots,(1,\ldots,1))}}{\lra}\: Q_\bm,
\quad
1 \times Q_\bn \:\sr{\kappa\times\id}{\lra} Q_1 \times Q_\bn \:\sr{\gamma_{1,((\bn))}}{\lra}\: Q_\bn
\end{align}
are identified with the identity morphism via the canonical isomorphism $Q_\bn\times 1^{|\bm|}\cong Q_\bn$.
\break\null\hfill$\triangle$
\end{itemize}
\end{definition}

\subsection{Examples of relative 2-operads}
\label{ss:examples}

We now turn to two examples of relative 2-operads.
The first, $(\twoO_\bn)$, is an illustrative example intended as a warm-up.
It is a 2-operad in spaces relative to the little intervals operad, and each $\twoO_\bn$ sits inside of the arity-$|\bn|$ space in the little squares operad.
The second is the 2-associahedra $(W_\bn)$, which form a 2-operad of posets relative to the operad of associahedra.
This is our marquee example; in fact, it motivated the notion of relative 2-operads.
We expect these two relative 2-operads to be homotopy equivalent in an appropriate sense not yet defined, once we pass to a topological realization of the latter.

\subsubsection{$(\twoO_\bn)$}

We begin with the relative 2-operad $(\twoO_\bn)$.
Each $\twoO_\bn$ is a configuration space of $|\bn|$ disjoint rectangles inside a bounding square.
$\twoO_\bn$ is a subspace of the $|\bn|$-th space in the little 2-cubes operad, obtained by requiring certain of the rectangles to be horizontally aligned.
First, we recall the definition of the little intervals operad $\bigl(\oneO_r\bigr)$.

\begin{definition}
For any $r \geq 1$, define $\oneO_r$ to be the space of increasing linear embeddings of $r$ copies of $[0,1]$ into $[0,1]$, such that:
\begin{itemize}
\item The $r$ images are disjoint.

\item For any $i<j$, the image of the $i$-th interval is to the left of the image of the $j$-th.
\end{itemize}
$\bigl(\oneO_r\bigr)_{r\geq1}$ forms an operad, where the composition maps are defined by rescaling and inserting configurations as in Fig.~\ref{fig:intervals_ex}.
\null\hfill$\triangle$
\end{definition}

\begin{figure}[H]
\centering
\def\svgwidth{1.0\columnwidth}
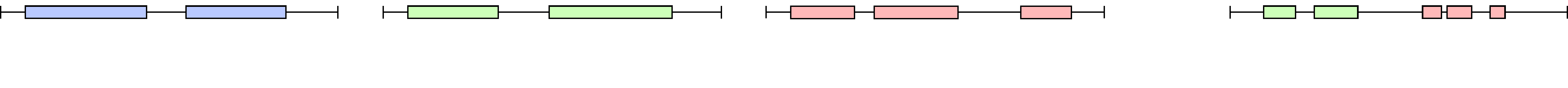
\caption{Here we illustrate the composition map $\gamma_{2,(2,3)}$ of $\bigl(\oneO_r\bigr)$.
It acts by linearly shrinking the second and third configurations and using them to replace the two intervals in the first configuration.
\label{fig:intervals_ex}}
\end{figure}

\noindent
We can now define $\bigl(\twoO_\bn\bigr)$, which forms a 2-operad relative to $\bigl(\oneO_r\bigr)$.

\begin{definition}
For any $r \geq 1$ and $\bn \in \bZ^r_{\geq0}\setminus\{\bzero\}$, define $\twoO_\bn$ to be the space of pairs $(2C,C)$, where $C$ is a configuration in $\oneO_r$ and where $2C$ is a collection of linear embeddings of $|\bn|$ copies of $[0,1]^2$ into $[0,1]^2$ satisfying the following properties:
\begin{itemize}
\item Each embedding is of the form $(x,y) \mapsto (ax+c,by+d)$ with $a,b > 0$.

\item The $|\bn|$ images are disjoint.

\item Reindex the embeddings by referring to the $(n_1+\cdots+n_{i-1}+j)$-th embedding as the $(i,j)$-th embedding.
Then we require that for any $i,j$, the postcomposition of the $(i,j)$-th embedding with the projection $\pr_1\colon \bR^2 \to \bR$ is equal to the $i$-th embedding in $C$.

\item For any $i$ and $j<j'$, the image of the $(i,j)$-th embedding lies below the image of the $(i,j')$-th.
\end{itemize}
$\bigl(\twoO_\bn\bigr)$ is a 2-operad relative to $\bigl(\oneO_r\bigr)$, with composition maps defined by rescaling and inserting configurations as in Fig.~\ref{fig:two-cubes_ex} and with projection $\twoO_\bn \to \oneO_r$ given by sending $(2C,C)$ to $C$.
\null\hfill$\triangle$
\end{definition}

\noindent
We denote an element $(2C,C)$ of $\twoO_\bn$ by a square above a horizontal interval, as shown below in Fig.~\ref{fig:two-cubes_ex}.
The interval is decorated by subintervals, which indicate the images of the embeddings in $C$.
The preimages of the intervals under the first projection $\pr_1\colon \bR^2 \to \bR$ are shown as subrectangles of height 1 in the square.
The images of the embeddings in $2C$ are denoted by subrectangles of the square of height less than 1.

\begin{figure}[H]
\centering
\def\svgwidth{1.0\columnwidth}
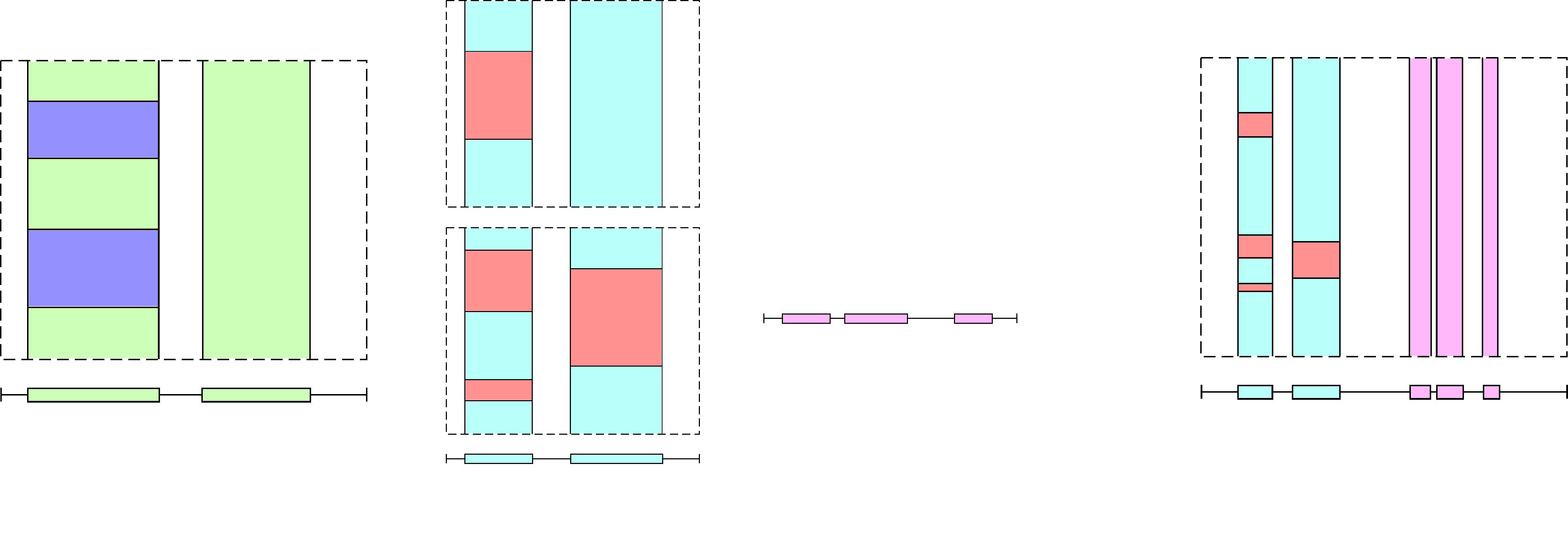
\caption{Here we illustrate one of the composition maps of $\bigl(\twoO_\bn\bigr)$.
It acts on the underlying configurations of intervals by $\gamma_{2,(2,3)}$, and on the configurations of rectangles by linearly shrinking the second and third configurations and using them to replace the two blue rectangles in the first configuration.
\label{fig:two-cubes_ex}
}
\end{figure}

\subsubsection{The 2-associahedral relative 2-operad $(W_\bn)$}

We begin by recalling the definition of the 2-associahedra $W_\bn$ from \cite{b:2ass}, verbatim.

\medskip

\noindent
{\bf Def.\ 3.1, \cite{b:2ass}.}
A \emph{stable tree-pair of type $\bn$} is a datum $2T = T_b \sr{f}{\to} T_s$, with $T_b, T_s, f$ described below:
\begin{itemize}
\item The \emph{bubble tree} $T_b$ is a planted ribbon tree whose edges are either solid or dashed, which must satisfy these properties:
	\begin{itemize}
		\item The vertices of $T_b$ are partitioned as $V(T_b) = V_\comp \sqcup V_\seam \sqcup V_\mk$, where:
			\begin{itemize}
				\item every $\alpha \in V_\comp$ has $\geq 1$ solid incoming edge, no dashed incoming edges, and either a dashed or no outgoing edge;
				\item every $\alpha \in V_\seam$ has $\geq 0$ dashed incoming edges, no solid incoming edges, and a solid outgoing edge; and
				\item every $\alpha \in V_\mk$ has no incoming edges and either a dashed or no outgoing edge.
			\end{itemize}
			We partition $V_\comp \eqqcolon V_\comp^1 \sqcup V_\comp^{\geq2}$ according to the number of incoming edges of a given vertex.
		\item ({\sc stability}) If $\alpha$ is a vertex in $V_\comp^1$ and $\beta$ is its incoming neighbor, then $\#\!\incom(\beta) \geq 2$; if $\alpha$ is a vertex in $V_\comp^{\geq2}$ and $\beta_1,\ldots,\beta_\ell$ are its incoming neighbors, then there exists $j$ with $\#\!\incom(\beta_j) \geq 1$.
	\end{itemize}
	\item The \emph{seam tree} $T_s$ is an element of $K_r^\tree$, i.e.\ the poset of planted ribbon trees with $r$ leaves.
	\item The \emph{coherence map} is a map $f\colon T_b \to T_s$ of sets having these properties:
		\begin{itemize}
			\item $f$ sends root to root, and if $\beta \in \incom(\alpha)$ in $T_b$, then either $f(\beta) \in \incom(f(\alpha))$ or $f(\alpha) = f(\beta)$.
			\item $f$ contracts all dashed edges, and every solid edge whose terminal vertex is in $V_\comp^1$.
			\item For any $\alpha \in V_\comp^{\geq2}$, $f$ maps the incoming edges of $\alpha$ bijectively onto the incoming edges of $f(\alpha)$, compatibly with $<_\alpha$ and $<_{f(\alpha)}$.
			\item $f$ sends every element of $V_\mk$ to a leaf of $T_s$, and if $\lambda_i^{T_s}$ is the $i$-th leaf of $T_s$, then $f^{-1}\{\lambda_i^{T_s}\}$ contains $n_i$ elements of $V_\mk$, which we denote by $\mu_{i1}^{T_b},\ldots,\mu_{in_i}^{T_b}$.
		\end{itemize}
\end{itemize}
We denote by $W_\bn^\tree$\label{p:Wntree} the set of isomorphism classes of stable tree-pairs of type $\bn$.
Here an isomorphism from $T_b \sr{f}{\to} T_s$ to $T_b' \sr{f'}{\to} T_s'$ is a pair of maps $\varphi_b\colon T_b \to T_b'$ and $\varphi_s\colon T_s \to T_s'$ that fit into a commutative square in the obvious way and that respect all the structure of the bubble trees and seam trees.
\null\hfill$\triangle$

Next, we recall that for every stable tree-pair $2T$ there is an inclusion of the following form:
\begin{align}
\Gamma_{2T} \colon \prod_{
{\alpha \in V_\comp^1(T_b)}
\atop
{\incom(\alpha)=(\beta)}
} W_{\#\!\incom(\beta)}^\tree
\times
\prod_{\rho \in V_\inte(T_s)} \prod^{K_{\#\!\incom(\rho)}}_{
{\alpha\in V_\comp^{\geq2}(T_b)\cap f^{-1}\{\rho\}}
\atop
{\incom(\alpha)=(\beta_1,\ldots,\beta_{\#\!\incom(\rho)})}
}
\hspace{-0.25in} W^\tree_{\#\!\incom(\beta_1),\ldots,\#\!\incom(\beta_{\#\!\incom(\rho)})}
\hra W_\bn^\tree.
\end{align}
This map is defined in \cite[Def.-Lem.\ 4.4]{b:2ass}.
While the complete setup for this map would take us too far afield, in Fig.~\ref{fig:2op_composition} below we will give an illustrative example.
First, for any $r, s_1,\ldots, s_r \geq 1$, $\bm \in \bZ^r_{\geq0}\setminus\{\bzero\}$, and $\bn_i^1,\ldots,\bn_i^{m_i} \in \bZ_{\geq0}^{s_i}$, we define $2T_{\bm,(\bn_i^a)}$ to be the following tree-pair in $W_{\sum_a \bn_1^a,\ldots,\sum_a \bn_r^a}$:
\begin{figure}[H]
\centering
\def\svgwidth{0.55\columnwidth}
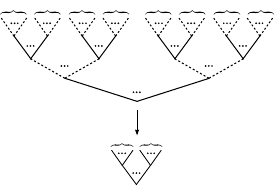
\end{figure}
\noindent
The associated composition map maps between the following posets:
\begin{align}
\Gamma_{2T_{\bm,(\bn_i^a)}}
\colon
W_\bm
\times
\prod_{1\leq i\leq r} \prod^{K_{s_i}}_{1\leq a\leq m_i} W_{\bn_i^a}
\to
W_{\sum_a \bn_1^a,\ldots,\sum_a \bn_r^a}.
\end{align}
On the level of seam trees, this map is the ordinary operadic composition in $K_r$.
On the level of bubble trees, this map attaches the roots of the bubble trees of the elements in $\prod_{1\leq i\leq r} \prod^{K_{s_i}}_{1\leq a\leq m_i} W_{\bn_i^a}$ to the leaves of the bubble tree of the element of $W_\bm$.
We depict an example in the following figure:
\begin{figure}[H]
\centering
\def\svgwidth{1.0\columnwidth}
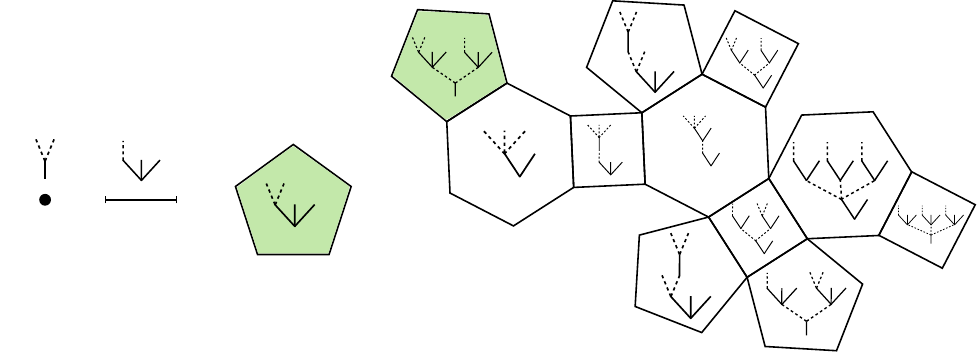
\caption{Each $W_\bn$ is an abstract polytope.
The target is the face lattice of a polyhedron; we depict its net on the right-hand side.
We label the top faces of each polytope by the bubble tree of the corresponding tree-pair; for reasons of space, we do not include the seam tree, and we do not label the positive-codimension faces.
The green pentagon in the codomain is the image of this composition map.
\label{fig:2op_composition}}
\end{figure}

We can finally exhibit the 2-associahedra as a 2-operad relative to the associahedra.
\begin{defprop}
\label{defprop:2Mn_is_2op}
Define $\Gamma_{\bm,(\bn^a_i)}
\colon
W_\bm \times \prod_{1\leq i\leq r} \prod^{K_{s_i}}_{1\leq a\leq m_i} W_{\bn_i^a} \to W_{\sum_a \bn_1^a,\ldots,\sum_a \bn_r^a}$ by setting $\Gamma_{\bm,(\bn^a_i)} \coloneqq \Gamma_{2T_{\bm,(\bn^a_i)}}$, where the latter map was defined in Def.-Lem.~4.3, \cite{b:2ass}.
Then $\bigl((K_r),(W_\bn)\bigr)$ with these composition maps forms a relative 2-operad in the category of posets.
\end{defprop}

\begin{proof}
{\sc(unit)} holds trivially, and {\sc(projections)} is equivalent to the observation that for any tree-pair $2T = (T_b \to T_s)$, $\Gamma_{2T}$ and $\gamma_{T_s}$ are intertwined by the projections.
{\sc(associative)} holds by a diagram chase similar to the one conducted in the proof of Def.~Prop.~\ref{defprop:Kr_is_an_operad} above.
We conduct this chase in the following figure.
Because this figure is so complex, we display only the bubble trees, rather than the full tree-pairs.
In fact, the diagram chase conducted in the proof of Def.~Prop.~\ref{defprop:Kr_is_an_operad} is identical to what happens to the seam trees in the following figure.
\begin{figure}[H]
\centering
\def\svgwidth{1.1\columnwidth}
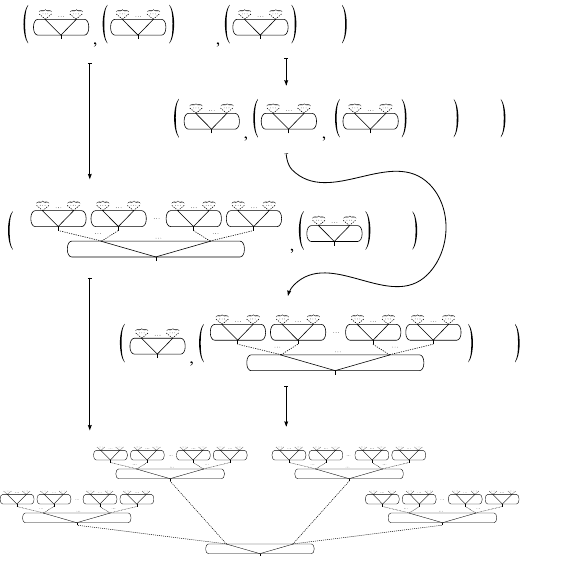
\end{figure}
\end{proof}

Similarly, $\bigl((\cM_r),(\ol{2\cM}_\bn)\bigr)$ is a relative 2-operad in \textsf{Top}.

\subsection{Relative 2-operads and Batanin's $2$-operads}
\label{ss:batanin}

In this subsection, we relate the notion of a relative 2-operad presented in Def.~\ref{def:2op} above to Batanin's theory of monoidal globular categories.
Specifically, we will show in Prop.~\ref{prop:connection_with_batanin} that relative 2-operads coincide with certain 2-operads in a certain globular category.

Before we can define the notion of a 2-operad, we must define $r$-globular categories, $r$-globular monoidal categories, collections in globular categories, and the monoidal structure on $\coll(\cC)$.
First, we recall the definition of an $r$-globular category:

\medskip

\noindent
{\bf Definition 2.1, \cite{ba:monoidal} (paraphrased).}
An $r$-globular category $\cC$ is a sequence of categories $(C_0,...,C_r)$ and functors $s_{k,l},t_{k,l}\colon C_k \to C_l$ for $k>l$ satisfying the following relations:
\begin{gather}
s_{l,m}\circ s_{k,l}=s_{k,m},
\quad 
t_{l,m}\circ t_{k,l}=t_{k,m},
\\
s_{l,l-1}\circ t_{l+1,l}=s_{l+1,l-1},
\quad
s_{l,l-1}\circ t_{l+1,l}=s_{l+1,l-1},
\quad
t_{l,l-1}\circ s_{l+1,l}=t_{l+1,l-1}.
\nonumber
\end{gather}
\noindent
If $\cC$ is an $r$-globular category for $0 < r < \infty$, and $s < r$, we denote by $C_{<s}$ the globular category obtained from $\cC$ by omitting all the categories $C_t$ for $s<t\leq r$.
\null\hfill$\triangle$

\medskip

\noindent
The most relevant case of a globular category for us is the case $r=2$.
A 2-globular category consists of a category $C_0$ of ``objects'', a category $C_1$ of ``morphisms'', and a category $C_2$ of ``2-morphisms''.
Note that in this definition, we assume no composition rule on $C_1$ and $C_2$. 

We will also need the notion of a monoidal structure on an $n$-global category (which we could call a ``composition rule'', in the terminology of the previous paragraph).

\medskip

\noindent
{\bf Definitions 2.3 and 2.6, \cite{ba:monoidal} (paraphrased).}
An $r$-globular monoidal category consists of an $r$-globular category $\cC$ together with various ``relative composition rules'' $\otimes_{i,j}\colon C_i \times_{s_{i,j},t_{i,j}} C_i \to C_i$ for $j<i$ and ``relative identity maps'' $1_{i,j}\colon C_i \to C_j$ for $i<j$, satisfying certain associativity and unity relations.
An \emph{augmented} $r$-globular monoidal category is an $r$-globular monoidal category together with a compatible collection of monoidal structures $(\otimes_{i,-1}\colon C_i \times C_i \to C_i)_{0 \leq i \leq r}$.
\null\hfill$\triangle$

\medskip


The final preliminary definition is that of a \emph{collection} in an $r$-globular category.
To make this definition, we will need to use $\Tree$, the $r$-globular category of $r$-stage trees.
Namely, $\Tree_i$ is the category of functors $F\colon [i]^\op \to \Delta$ where $[i]$ is the linearly-ordered set of size $i+1$ and $\Delta$ is the category of ordered sets.
Geometrically, we think of $F(j)$ as the set of vertices of distance $j$ from the root of the tree.
The unique morphism $F(j)\to F(j-1)$ assigns to a vertex $v\in F(j)$ its immediate parent $p(v)$ in the tree $F$.
Then $\Tree \coloneqq \bigl(\Tree_i\bigr)_{i\ge 0}$ forms an $\omega$-globular category, with the source and target maps both induced from the prefix embedding $[j]\to [i]$ for $j<i$.
Geometrically, this means that $s_{i,j}(F) = t_{i,j}(F)$ is obtained from $F$ by chopping off all vertices of $F$ with height greater than $j$.
Moreover, $\Tree$ admits a globular monoidal structure with relative tensor product obtained from the monoidal structure on $\Delta$ by concatenation, and the unit maps $1_{i,j}$ obtained by declaring an $j$-stage tree to be $i$-stage tree with no leaves of height greater than $j$.

We now come to the notion of a \emph{collection} in an $r$-globular category. 

\medskip

\noindent
{\bf Definition 6.1, \cite{ba:monoidal}.}
A collection in an $r$-globular category $\cC=\bigl(C_i\bigr)_{0\le i\le r}$ is a globular functor $\Tree_{\le r}\to C$.
\null\hfill$\triangle$

\medskip

\noindent
In other words, a collection $A$ assigns to each tree $T$ of height $i < r$ an object of $C_i$, thought of intuitively as the object of ``$T$-shaped operations''.
These choices are required to be compatible, in the sense that we must specify isomorphisms between $s_{i,i-1}A(T)$, $t_{i,i-1}A(T)$, and $A(\partial(T))$ in a compatible way.   

Suppose that $\cC$ is an augmented monoidal $r$-globular category $\cC$ with finite globular coproducts which distribute over all tensors in $\cC$.
In this situation, Batanin equips $\coll(\cC)$ with a monoidal structure (\cite[Thm.\ 6.1]{ba:monoidal}).
(The notions of globular coproduct and the distributivity of such products are defined in \cite[Def.\ 5.2]{ba:monoidal} and \cite[Def.\ 5.4]{ba:monoidal}, respectively.)
We shall briefly explain the underlying binary tensor product in the case $r=2$, which is the case relevant to us.
To do so, we will need the notion of a \emph{2-tree}.

\begin{definition}
A 2-tree of height $i$ is a tree $F\colon[i]^\op \to \Delta$ and an assignment $v\mapsto G_v$ from vertices of $F$ to trees, such that $G_v$ is a $\rm{ht}(v)$-stage tree, and such that for every $v\in F(j)$ we have $G_{p(v)} = s_{j,j-1}(F_v)$.
\null\hfill$\triangle$  
\end{definition}      

\noindent
Note that to specify a 2-tree, it suffices to specify its values on the leaves of $F$; all other junctions in $F$ only impose constraints on these choices.

Let $2\Tree = \bigl(2\Tree_i\bigr)_{i>0}$ denote the 2-globular category of $2$-trees, with its obvious globular structure.
There is then a functor $\delta\colon 2\Tree \to \Tree$, defined like so:
\begin{gather}
\delta\Bigl(F,\bigl(G_v\bigr)_{v\in \coprod_jF(j)}\Bigr)
\coloneqq
\text{the tree with vertices }
\coprod_j \coprod_{v\in F(j)} G_v(j),
\text{ such that } p(v,u)=(p(v),p(u)).
\end{gather}
(In the definition of $\delta$, we have identified $p(u)$ with an element of $G_{p(v)}$ via the isomorphism $G_{p(v)}\cong \partial G_v$.)
Using this functor, we can equip the $r=2$ case of $\coll(\cC)$ with a monoidal structure.
We define the binary tensor product on $\coll(\cC)$ to be the following operation:
\begin{align}
\label{eq:tensor_on_collections}
(A\otimes B) (F)
\coloneqq
\coprod_{\delta(F',\{G_v\})=F} A(F') \otimes_{2,-1} \bigotimes_{
{2,0}
\atop
{v\in F'\colon \height(v)=1}
}
\left(\bigotimes_{
{2,1}
\atop
{\height(u)=2,p(u)=v}
}
B(G_v) \right).
\end{align}
In the general case, one recursively replaces the value at a vertex of the base-tree $F'$ by the relative tensor product of the values at its immediate descendents, where the relative tensor product is the one that matches the types of the objects decorating the relevant vertices.

We are finally ready to define the notion of an $r$-operad. 
\begin{definition}
Let $\cC$ be an $r$-globular monoidal category. An $r$-operad in $\cC$ is an associative algebra in $\coll(\cC)$.
\null\hfill$\triangle$
\end{definition}

\noindent
In order to identify the notion of a relative 2-operad with a certain class of 2-operads, we will define a certain $r$-globular category $\cC$.
Let $\cD$ be a category with finite limits, endowed with a distinguished class of morphisms $F$ that is closed under composition, contains all isomorphisms, and is closed under pullbacks.
Set $\cC\coloneqq\{C_0,C_1,C_2\}$, where $C_0\coloneqq \pt$ and $C_1 \coloneqq \cD$, and where $C_2 \coloneqq \spn_F(\cD)$ is the category $X$ of spans from $Y$ to $Z$ in $\cD$, where both maps are in $F$, and with morphisms the morphisms of diagrams.
The source and target maps are just given by $X$ and $Y$ respectively.
This 2-globular category $\cC$ is endowed with a globular monoidal structure given as follows. The relative tensor product of $C_1$ over $C_0$ is given by the product in $\cD$.
The relative tensor product of $C_2$ over $C_0$ is given by the monoidal structure on $\spn_F{\cD}$ induced from the product in $\cD$.
Finally, the relative tensor product of $C_2$ over $C_1$ is given by the composition of spans.

Finally, we can relate our notion of relative 2-operad to the notion of $r$-operad. 
\begin{proposition}
\label{prop:connection_with_batanin}
Let $\cD$ be a category with finite limits.
A relative 2-operad in $\cD$ is a 2-operad in the globular category $\cC$ as above, for which the source and target maps coincide.
\end{proposition}

\noindent
{\it Proof sketch.}
We show how to get a relative 2-operad from a 2-operad in the globular category $\cC$ for which the source and target maps coincide.
The other direction is similar.
Given a 2-operad $A$ in $\cC$, we need to supply a relative 2-operad 
\begin{align}
\Bigl((P_r)_{r\geq 1}, (Q_\bm)_{\bm \in \bZ^r_{\geq0}\setminus\{\bzero\},r \geq 1}\Bigr),
\end{align}
Denote by $T_r$ the tree consisting of a root and $r$ leaves; for $\bm \in\bZ^r_{\geq0}\setminus\{\bzero\}$, denote by by $T_{r,(\bm_i)}$ the tree with $r$ leaves $v_1, \ldots, v_r$ of height 1 and $\bm_i$ leaves of degree 2 with parent $v_i$, as pictured here:
\begin{figure}[H]
\centering
\def\svgwidth{0.3\columnwidth}
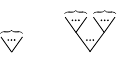
\end{figure}
\noindent
We set $P_r \coloneqq A(T_r)\in \cC_1 =\cD$.
By definition, the object $A(T_{r,(\bm_i)})$ is an object of $\cD$ which we take to be $Q_\bm$, endowed with a map
\begin{align}
A(T_{r,(\bm_i)}) \simeq Q_\bm
\to
P_r \simeq A(T_r)
\end{align}
which we take to be the structure map $\pi_\bm$ of the relative 2-operad. 
	
The multiplication maps $A\otimes A \to A$ now give us the structure maps $\gamma_{r,(s_i)}\colon P_r \times \prod_i P_{s_i} \to P_{\sum_i s_i}$ as follows.
For the height-1 tree $T_{\sum_i s_i}$, $A(T_r) \times \prod_i A(T_{s_i}) \cong P_r \times \prod_i P_{s_i}$ is a summand of $(A\otimes A)(T_{\sum_i s_i})$ and so the multiplication map of $A$ provides in particular a map $\gamma_{r,(s_i)}$ as above.
For the maps $\Gamma_{\bm,(\bn^a_i)}$, the tensor product formula \eqref{eq:tensor_on_collections} shows that the height 2 tree $T_{\sum_i s_i,(\sum_a\bn^a_i)}$, which realizes as the diagonal of a 2-stage tree.
Hence the map
\begin{align}
(A\otimes A)\bigl(T_{\sum_i s_i,(\sum_a\bn^a_i)}\bigr)
\to
A (T_{\sum_i s_i,(\sum_a\bn^a_i)})
\end{align}         
gives in particular a map $\Gamma_{\bm,(\bn^a_i)}$ as in \eqref{eq:Gamma_def}, in which the fiber products in the source correspond to the operation $\displaystyle{\bigotimes_{1,2}}$ and the products to $\displaystyle{\bigotimes_{0,2}}$.
Finally, the unit of $A$ gives the unit of the relative 2-operad.
The associativity constraints of $\Gamma_{\bm,(\bn^a_i)}$ and $\gamma_{r,(s_i)}$ (cf.\ \eqref{eq:operad_ass}, and \eqref{eq:relative_2-operad_Gamma}) now correspond to the associativity of the multiplication of $A$, and the unitality of the operad to the unitality of $A$
Finally, \eqref{eq:Gamma_gamma_coherence} corresponds to the fact that the multiplication happens in spans, hence is compatible with the projection to the base of the span.
\null\hfill\qed

\begin{remark}
One can organize all relative 2-operads into a category in such a way that Proposition~\ref{prop:connection_with_batanin} provides a fully faithful embedding of relative 2-operads into Batanin's category of 2-operads.
\null\hfill$\triangle$
\end{remark}

\section{Categories over relative 2-operads}
\label{s:cats_over_2ops}

In this final section, we turn to the notion of a category over a relative 2-operad. 
In particular, in \S\ref{ss:A-infinity-2-tilde-spaces} we consider algebras over $(\twoO_\bn)$, which we call \emph{$\wt{(A_\infty,2)}$-spaces}; in \S\ref{ss:A-infinity-2-categories}, we consider categories over $(\ol{2\cM}_\bn)$, which we call \emph{$(A_\infty,2)$-categories}.
The definition of $(A_\infty,2)$-categories is the main contribution of this paper, and is a necessary part of the first author's project to construct the symplectic $(A_\infty,2)$-category.
In \S\ref{ss:A-infinity-2-tilde-spaces} we prove Prop.~\ref{prop:theta}, which asserts that from a map $A \to X$ of pointed spaces we can construct a $\wt{(A_\infty,2)}$-space $\theta(A\to X)$.
This provides a collection of examples of algebras over a relative 2-operad.

We note that just as the notion of a relative 2-operad can be rephrased in terms of Batanin's theory of $n$-operads (as discussed in \S\ref{ss:batanin}), the notion of a category over a relative 2-operad can be reformulated in Batanin's language.
For further details of the part of Batanin's theory relevant to this section, we direct the reader to \cite{batanin_markl}.

\subsection{The definition of a category over a relative 2-operad}

We shall now define the notion of a category over a relative 2-operad.
We note that there are other approaches to operadic higher category theory, see e.g.\ as in \cite{ba:symm,cheng:higher categories}; the approach we describe here is suited to the first author's ongoing project to construct the symplectic $(A_\infty,2)$-category, as described in \S\ref{s:intro}.
Another relevant construction is Tamarkin's homotopy 2-category of dg-categories, as in \cite[\S5.3]{tamarkin:what_do_dg-categories}; we expect that this fits into the formalism of $R$-linear $(A_\infty,2)$-categories constructed in the current paper.

Recall the well-known (see e.g.\ \cite[Def.~4]{may:categories}) notion of a category over an operad:

\begin{definition}
Let $O = (P_r)_{r\geq1}$ be an operad in a category $\cC$ with products, considered as a symmetric monoidal category using the Cartesian symmetric monoidal structure.
Recall that a \emph{(nonunital) category over $O$} consists of a set of objects $\Ob$ and, for every $x,y \in \Ob$, a morphism object $\Mor(x,y) \in \cC$, 
together with source and target maps $s,t\colon \Mor\to \Ob$.

The pair $(\Ob,\Mor)$ is equipped with higher composition maps of the form
\begin{align}
c_r\colon P_r\times \Mor(x_0,x_1)\times\cdots\times \Mor(x_{n-1},x_n)\to \Mor(x_0,x_n),
\end{align}
which are associative in the sense that the following diagram commutes for every choice of (a) a sequence of objects $x_0,\ldots,x_r \in \Ob$ and (b) further sequences $y_0^i=x_{i-1},y_1^i,\ldots,y_{s_i}^i=x_i$ for every $i$:
\begin{align}
\label{eq:assoc_for_cat}
\xymatrix{
P_{\sum s_i}\times \prod\limits_{{1\leq i\leq r}\atop{1\leq j\leq s_i}} \Mor(y_{j-1}^i,y_j^i)
\ar[rr]^{\hspace{0.35in}c_{\sum_i s_i}}
&&
\Mor(x_0,x_r)
\\
P_r\times\prod\limits_{1\leq i\leq r} P_{s_i}\times\prod\limits_{{1\leq i\leq r}\atop{1\leq j\leq s_i}} \Mor(y_{j-1}^i,y_j^i)
\ar[u]^{\gamma_{r,(s_i)}\times\id}
\ar[d]_\simeq
&&
\\
P_r\times\prod\limits_{1\leq i\leq r}\Bigl(P_{s_i}\times\prod\limits_{1\leq j\leq s_i} \Mor(y_{j-1}^i,y_j^i)\Bigr)
\ar[rr]_{\hspace{0.5in}\id\times(c_{s_i})}
&&
P_r \times \prod\limits_{1\leq i\leq r} \Mor(x_{i-1},x_i)
\ar[uu]_{c_r}
}
\end{align}
\null\hfill$\triangle$
\end{definition}

%

\begin{remark}
Under the identification of a relative 2-operad with a 2-operad in the sense of Batanin, one can further identify a category over a relative 2-operad with an algebra over the associated 2-operad, in the sense of \cite[Def.\ 7.3]{ba:monoidal}.
Note that this a \emph{different} convention from the one we use here: for us, an algebra over a relative 2-operad is simply a category over a relative 2-operad, which has only a single object.
\null\hfill$\triangle$
\end{remark}

To adapt this to the notion of a 2-category over a relative 2-operad $2O$, we just mimic this construction and add 2-morphisms to the story. 

\begin{definition}
\label{def:2cat}
Let $2O=\bigl((P_r)_{r\geq 1}, (Q_\bm)_{\bm \in \bZ^r_{\geq0}\setminus\{\bzero\}}\bigr)$ be a relative 2-operad in a category $\cC$ with finite limits. A (non-unital) category over $2O$ consists of the following data: 
\begin{itemize}
\item A set of objects $\Ob$.
 
\item For each $x,y\in \Ob$, an object $\Mor(x,y) \in \cC$, which we think of as morphisms from $x$ to $y$.

\item For each $x,y\in \Ob$, an object $2\Mor(x,y) \in \cC$, which we think of as 2-morphisms over $x,y$.

\item Source and target morphisms $s,t \colon 2\Mor(x,y)\to \Mor(x,y)$.

\item Composition laws: for each $x_0,\ldots,x_r\in \Ob$ a morphism
\begin{align}
c_r\colon P_r \times \prod_{j=1}^r \Mor(x_{j-1},x_j) \to \Mor(x_0,x_r).
\end{align}

\item 2-composition laws: For each $x_0, \ldots, x_r\in \Ob$ and each $\bm \in \bZ^r_{\geq0}\setminus\{\bzero\}$, a morphism
\begin{align}
2c_\bm\colon Q_\bm \times \prod_{j=1}^r 2\Mor(x_{j-1},x_j)^{\times_{\Mor(x_{j-1},x_j)} m_j} \to 2\Mor(x_0,x_r),
\end{align}
where $2\Mor(x_{j-1},x_j)^{\times_{\Mor(x_{j-1},x_j)} m_j}$ (slightly abusively) denotes the fiber product
\begin{align}
2\Mor(x_{j-1},x_j)^{\times_{\Mor(x_{j-1},x_j)} m_j}
\coloneqq
\underbrace{2\Mor(x_{j-1},x_j) \:{}_s\!\times_t \cdots \:{}_s\!\times_t 2\Mor(x_{j-1},x_j)}_{m_j}.
\end{align}
\end{itemize}

\noindent
We require these data to satisfy the following conditions: 
\begin{itemize}
\item The data $(\Ob,\Mor,c_r)$ is a category over $(P_r)_{r\geq 1}$.

\item The 2-composition must be associative, in the sense that the following diagram must commute, for every choice of (a) a sequence of objects $x_0,\ldots,x_r \in \Ob$ and (b) further sequences $y_0^i=x_{i-1},y_1^i,\ldots,y_{s_i}^i=x_i$ for every $i$:
\begin{align}
\label{eq:assoc_for_2cat}
\xymatrix{
Q_{\sum\limits_a\bn_1^a,\ldots,\sum\limits_a\bn_r^a} \times \prod\limits_{{1\leq i\leq r}\atop{1\leq j\leq s_i}} 2\Mor(y_{j-1}^i,y_j^i)^{\times_{\Mor(y_{j-1}^i,y_j^i)}\sum\limits_a n_{ij}^a} \ar@/^3pc/[rrd]^{\hspace{1in}2c_{\sum\limits_a\bn_1^a,\ldots,\sum\limits_a\bn_r^a}} &&
\\
Q_\bm\times\prod\limits_{1\leq i\leq r}\prod\limits_{1\leq a\leq m_i}^{P_{s_i}} Q_{\bn_i^a}\times\prod\limits_{{1\leq i\leq r}\atop{1\leq j\leq s_i}} 2\Mor(y_{j-1}^i,y_j^i)^{\times_{\Mor(y_{j-1}^i,y_j^i)}\sum\limits_a n_{ij}^a} \ar[u]^{\Gamma_{\bm,(\bn_i^a)}\times\id}\ar[d]_\simeq && 2\Mor(x_0,x_r)
\\
Q_\bm\times\prod\limits_{1\leq i\leq r}\prod\limits_{1\leq a\leq m_i}^*\Bigl(Q_{\bn_i^a}\times\prod_{1\leq j\leq s_i} 2\Mor(y_{j-1}^i,y_j^i)^{\times_{\Mor(y_{j-1}^i,y_j^i)}n_{ij}^a}\Bigr) \ar[d]_{\id\times(2c_{\bn_i^a})} &&
\\
Q_\bm \times \prod\limits_{1\leq i\leq r} 2\Mor(x_{i-1},x_i)^{\times_{\Mor(x_{i-1},x_i)}m_i} \ar@/_4pc/[rruu]_{\hspace{0.35in}2c_\bm} &&
}
\end{align}
The asterisk appearing above the product sign on the left indicates that we are taking an appropriate fiber product so that the image under $(2c_{\bn_i^a})_a$ is in $2\Mor(x_{i-1},x_i)^{\times_{\Mor(x_{i-1},x_i)}m_i}$.

\item The composition of 2-morphisms and that of 1-morphisms must be compatible in the following sense.
Let $s_n,t_n \colon 2\Mor(x,y)^{\times_{\Mor(x,y)} n}\to \Mor(x,y)$ denote the compositions
\begin{gather}
s_n\colon\:2\Mor(x,y)^{\times_{\Mor(x,y)} n}\stackrel{\proj_1}{\lra} 2\Mor(x,y)\stackrel{s}{\lra} \Mor(x,y),
\\
t_n\colon\: 2\Mor(x,y)^{\times_{\Mor(x,y)} n}\stackrel{\proj_n}{\lra} 2\Mor(x,y)\stackrel{t}{\lra} \Mor(x,y),
\nonumber
\end{gather}
where $\proj_j$ is the projection to the $j$-th coordinate.
Then we require the following diagram to commute:
\begin{align}
\xymatrix{
Q_{\bm}\times \prod_j 2\Mor(x_{j-1},x_j)^{\times_{\Mor(x_{j-1},x_j)} m_j} \ar^{\hspace{0.75in}2c_\bm}[rr]\ar_{\bigl(\prod_j s_{\bm_j},\prod_j t_{\bm_j}\bigr)}[d] && 2\Mor(x_0,x_r) \ar^{(s,t)}[d]\\ 
(P_r \times \prod_j \Mor(x_{j-1},x_j))^2 \ar_{\hspace{0.25in}(c_r,c_r)}[rr]                                                                                 && \Mor(x_0,x_r)^2
}
\end{align}
\null\hfill$\triangle$
\end{itemize}
\end{definition}

\subsection{$\wt{(A_\infty,2)}$-spaces}
\label{ss:A-infinity-2-tilde-spaces}

We now define the notion of an $\wt{(A_\infty,2)}$-space and prove Prop.~\ref{prop:theta}.

\begin{definition}
	An $\wt{(A_\infty,2)}$-space is an algebra over $\bigl((\oneO_r),(\twoO_\bn)\bigr)$, i.e.\ a pair of spaces $2Y \sr{s,t}{\rightrightarrows} Y$ such that $Y$ is an $A_\infty$-algebra, and $2Y$ is equipped with composition maps
	\begin{align}
	\label{eq:Ainf_2_space}
	\twoO_\bn \times \underbrace{2Y \,_t\!\times_s \cdots \,_t\!\times_s 2Y}_{n_1} \times \cdots \times \underbrace{2Y \,_t\!\times_s \cdots \,_t\!\times_s 2Y}_{n_r}
	\to
	2Y
	\end{align}
	that satisfy suitable coherence conditions.
\null\hfill$\triangle$
\end{definition}

\begin{proposition}\label{prop:theta}
Fix a map $f\colon (A,q) \to (X,p)$ of pointed topological spaces.
Define a space $\theta(A \to X)$ by
\begin{align}
\label{eq:theta_def}
\theta(A \to X)
\coloneqq
\left\{
\left.\left(
\substack{
u\colon [0,1]^2 \to X
\\
\gamma_\pm\colon [0,1] \to A}
\right)
\:\right|\:
\substack{
u(-,0) = f\circ\gamma_-,
\\
u(-,1) = f\circ\gamma_+},
\:
\substack{
u(0,-) = p = u(1,-)
\\
\gamma_\pm(0) = q = \gamma_\pm(1)}
\right\},
\end{align}
and equip $\theta(A\to X)$ with maps $s,t \colon \theta(A\to X) \to \Omega A$ that send $(u,\gamma_+,\gamma_-)$ to $\gamma_-$ resp.\ $\gamma_+$.
Then the pair $\theta(A\to X) \sr{s,t}{\rightrightarrows} \Omega A$ is an $\wt{(A_\infty,2)}$-space.
\end{proposition}

\begin{proof}
To equip $\theta(A\to X)$ with the structure of an $(A_\infty,2)$-space, we must define composition maps as in \eqref{eq:Ainf_2_space} and verify that they satisfy the appropriate coherences.
We do so as follows.
For $2Y=\theta(A \stackrel{f}{\to}X)$, we define the map
\begin{align*}
\twoO_\bn \times \underbrace{2Y \,_t\!\times_s \cdots \,_t\!\times_s 2Y}_{n_1} \times \cdots \times \underbrace{2Y \,_t\!\times_s \cdots \,_t\!\times_s 2Y}_{n_r}
\to
2Y
\end{align*} 
like so:
\begin{itemize}
\item Picture the configuration in $\twoO_\bn$ as the unit square with height-1 green rectangles that contain blue subrectangles, as on the left of Fig.~\ref{fig:theta_comp}.

\item For every green strip, we are given a choice of a loop in $A$, and for each blue rectangle, we are given a choice of a triple $(u,\gamma_+,\gamma_-)$ as in \eqref{eq:theta_def}.
We think of it the latter a map from the considered blue rectangle to $X$ and two maps from the upper and lower edges to $A$, compatible in the obvious sense.

\item The fiber product exactly allows us to define a map from the unit square to $X$, as on the right in \eqref{eq:theta_def}.
\end{itemize}
Associativity is clear from the picture.
 
\begin{figure}[H]
\centering
\def\svgwidth{1.0\columnwidth}
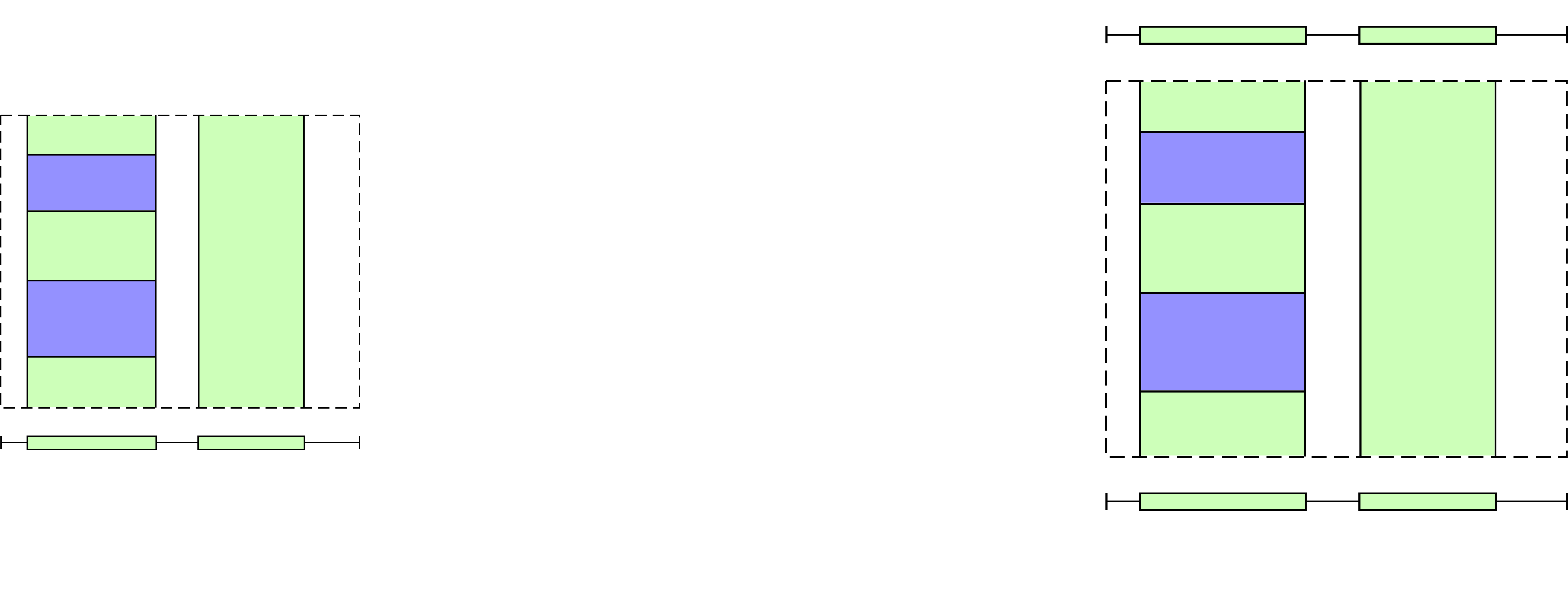
\caption{\label{fig:theta_comp}}
\end{figure}
\end{proof}

\subsection{$(A_\infty,2)$-categories}
\label{ss:A-infinity-2-categories}

The general definition of a category over a relative 2-operad specializes in the case of the 2-associahedral relative 2-operad to give the notion of an $(A_\infty,2)$-category over $\cC$.

\begin{definition}
An $(A_\infty,2)$-category is a category over the topological relative 2-operad $\bigl((\ol\cM_r),(\ol{2\cM}_\bn)\bigr)$.
\null\hfill$\triangle$
\end{definition} 

We would like to adapt Def.~\ref{def:2cat} to the case where the objects and 1-morphisms form an ordinary 1-category and the 2-morphisms are chain complexes over a ring, since this is the situation in the hypothetical $(A_\infty,2)$-category $\Symp$. 
In this situation, composition of 1-morphisms is independent of $P_r$, while composition of 2-morphisms is parametrized by the chain complex of $Q_\bm$.
We must be careful, because the collection $C_*(Q_\bm;R)$ is not a relative 2-operad in chain complexes, as the functor $C_*(-;R)$ is not limit preserving, and in particular does not behave well with respect to fibered products.
However, it is still possible to formulate a satisfactory definition, as we demonstrate below.
We will use suggestive notation for the 1-morphisms, which recalls the symplectic $(A_\infty,2)$-category.

\begin{definition}
Let $R$ be a ring. 
An \emph{$R$-linear category over a  relative 2-operad $\bigl((P_r),(Q_\bn)\bigr)$ in \textsf{Top}} consists of: 
\begin{itemize}
\item A category $(\Ob,\Mor,s,t)$.

\item For each pair of morphisms $L,K\colon M\to N$, a $\bZ$-graded complex of free $R$-modules $2\Mor(L,K)$.

\item Composition maps: for each $r\geq1$ and $\bm\in \bZ_{\geq0}^r\backslash \{\bzero\}$, for each sequence of objects $M_0,\ldots,M_r\in \Ob$, and for each collection of sequences $L_1^0,\ldots,L_1^{m_1},\ldots,L_r^0,\ldots,L_r^{m_r}$ with $L_i^j$ a morphism from $M_{i-1}$ to $M_i$, a composition map
\begin{align}
2c_\bm\colon C_*(Q_\bm)\otimes \bigotimes_{{1\leq i \leq r}\atop{1\leq j\leq m_i}}  2\Mor(L_i^{j-1},L_i^j)\to 2\Mor(L_1^0\circ\cdots\circ L_r^0,L_1^{m_1}\circ\cdots\circ L_r^{m_r}),
\end{align}
where $C_*(Q_\bm)$ denotes the complex of singular chains in $Q_\bm$ with coefficients in $R$.
\end{itemize}
\noindent
We require the composition maps to satisfy an associativity condition, expressed by the commutativity of the following diagram for every choice of (a) a sequence of objects $M_0,\ldots,M_r \in \Ob$, (b) further sequences $N_0^i=M_{i-1},N_1^i,\ldots,N_{s_i}^i=M_i$ for every $i$, and (c) 1-morphisms $L_{ij}^k\colon N_{j-1}^i\to N_j^i$:
\begin{gather}
\xymatrix{
C_*\bigl(Q_{\sum\limits_a\bn_1^a,\ldots,\sum\limits_a\bn_r^a}\bigr) \otimes \bigotimes\limits_{{1\leq i\leq r,1\leq j\leq s_i}\atop{1\leq k\leq \sum\limits_a n_{ij}^a}} \!2\Mor\bigl(L_{ij}^{k-1},L_{ij}^k\bigr)
\ar@/^6pc/[rdd]
&
\\
C_*\Bigl(Q_\bm\times\!\!\!\prod\limits_{1\leq i\leq r}\prod\limits_{1\leq a\leq m_i}^{P_{s_i}} Q_{\bn_i^a}\Bigr)\otimes\!\!\!\!\!\!\bigotimes\limits_{{1\leq i\leq r,1\leq j\leq s_i}\atop{1\leq k\leq \sum\limits_a n_{ij}^a}} \!\!\!\!\!\!2\Mor\bigl(L_{ij}^{k-1},L_{ij}^k\bigr) \ar[d]\ar[u]
&
\\
C_*(Q_\bm)\otimes\bigotimes\limits_{{1\leq i\leq r}\atop{1\leq a\leq m_i}} \Bigl(C_*(Q_{\bn_i^a})\otimes\!\!\bigotimes\limits_{{1\leq j\leq s_i}\atop{\sum\limits_{1\leq b< a} \!\!\!\!\! n_{ij}^b < k \leq \!\!\! \sum\limits_{1\leq b\leq a} \!\!\! n_{ij}^b}}\!\!\!2\Mor\bigl(L_{ij}^{k-1},L_{ij}^k\bigr)\Bigr) \ar[d]
&
\text{target}
\\
C_*(Q_\bm) \otimes \bigotimes\limits_{{1\leq i\leq r}\atop{1\leq a\leq m_i}} 2\Mor\Bigl(L_{i1}^{\sum\limits_{1\leq b< a} n_{i1}^b}\circ\cdots,L_{i1}^{\sum\limits_{1\leq b\leq a} n_{i1}^b}\circ\cdots\Bigr)
\ar@/_3pc/[ru]
&
}
\\
\nonumber
\\
\hspace{-2in}
\text{target}
\coloneqq
C\!\bigl(L_{11}^0\circ\cdots\circ L_{1s_1}^0\circ\cdots\circ L_{r1}^0\circ\cdots\circ L_{rs_r}^0,
\nonumber
\\
\hspace{1in}
L_{10,11}^{\sum_a n_{11}^a}\circ\cdots\circ L_{1(s_1-1),1s_1}^{\sum_a n_{1s_1}^a}\circ\cdots\circ L_{r0,r1}^{\sum_a n_{r1}^a}\circ\cdots\circ L_{r(s_r-1),rs_r}^{\sum_a n_{rs_r}^a}\bigr).
\nonumber
\end{gather}
\null\hfill$\triangle$
\end{definition}

\begin{remark}
Note that in the middle vertical map we implicitly use the swap maps of the tensor product of free modules.
We also use the natural, strictly associative map \[C_*(X\times_Z Y;R) \to C_*(X;R)\otimes C_*(Y;R)\] which is the composition of the map induced from the inclusion $X\times_Z Y\to X\times Y$ and the Alexander--Whitney map.
\end{remark}

We finally come to the definition that is one of the main contributions of this paper:

\begin{definition}
\label{def:A_infty_2_cat}
An \emph{$R$-linear $(A_\infty,2)$-category} is an $R$-linear category over the relative 2-operad $\bigl((\ol\cM_r),(\ol{2\cM}_\bn)\bigr)$.
\null\hfill$\triangle$
\end{definition} 

Having all these definitions and flavors of 2-categories over relative 2-operads, we can of course define algebras over a relative 2-operad.
These are just categories with a single object. 

\begin{definition}
Let $2O=\big(P_r,Q_\bm\big)$ be a relative 2-operad.
\begin{enumerate}
\item An \emph{algebra over $2O$} is a category over $O$ with a single object.

\item An \emph{$R$-linear algebra over $2O$} is an $R$-linear category over $2O$ with a single object.
\null\hfill$\triangle$
\end{enumerate}
\end{definition}

\subsection*{Acknowledgments}

Kevin Costello suggested that the definition of an $(A_\infty,2)$-category does not have to be based on a cellular model of $C_*(\ol{2\cM}_\bn)$, which helped the first author arrive at the relative 2-operadic structure of the 2-associahedra.
Paul Seidel suggested that the first author think about Rmk.~1.2.1 in \cite{el}, which led to Prop.~\ref{prop:theta}.
Jacob Lurie pointed out that the definition of an $(A_\infty,2)$-space should require two maps to the underlying $A_\infty$-space, not just one.
Discussions with Michael Batanin provided useful context.
The comments of an anonymous editor led the authors to analyze the connection with Batanin's work, which resulted in the subsection \S\ref{ss:batanin}.
The first author thanks Mohammed Abouzaid and Robert Lipshitz for their encouragement.

N.B.\ was supported by an NSF Mathematical Sciences Postdoctoral Research Fellowship and a Schmidt Fellowship.
S.C.\ was supported by the Adams Fellowship of the Israeli Academy of Science and Humanities.

\end{document}